\documentclass{article}
\usepackage{amssymb}
\usepackage{amsfonts}
\usepackage{amsmath}
\usepackage{chapterbib}

\newtheorem{theorem}{Theorem}

\newtheorem{definition}[theorem]{Definition}

\newtheorem{lemma}[theorem]{Lemma}

\newtheorem{proposition}[theorem]{Proposition}
\newtheorem{remark}[theorem]{Remark}

\newenvironment{proof}[1][Proof]{\noindent\textbf{#1.} }{\ \rule{0.5em}{0.5em}}
\input{tcilatex}
\begin{document}

\title{Analytic Properties of the Conformal Dirac Operator on the Sphere}
\author{Brett Pansano \\
Northwest Arkansas Community College}
\maketitle

\begin{abstract}
In this paper the conformal Dirac operator on the sphere is defined to be
operating on the space of $L^{2}\left( S^{n}\right) $ Clifford
algebra-valued functions. The spinorial Laplacian of order d is defined and
used to establish Sobolev embedding theorems, and an extension of the
conformal Dirac operator to the Sobolev space setting. Interpolation of
polynomials in Clifford analysis is left for further investigation.
\end{abstract}

\section{\protect\bigskip Preliminaries}

A Clifford algebra $Cl_{n}$ can be generated from the standard orthogonal
space $%
\mathbb{R}
^{n}$ with the negative definite inner product, that is, $x^{2}=-\left\Vert
x\right\Vert ^{2}.$ Then a basis is given by the set of ordered products $%
e_{A}=e_{j_{1}}\cdot \cdot \cdot e_{j_{k}},$ where $j_{1}<\cdot \cdot \cdot
<j_{k}$ and $A=\left\{ e_{j_{1}},...,e_{j_{k}}\right\} $ is a subset of $%
\mathbb{N}
.$ For $A=\varnothing ,e_{A}=1.$which is the identity element in $Cl_{n}.$ A
consequence of the equality $x^{2}=-\left\Vert x\right\Vert ^{2}$ is that $%
\left\{ e_{1},...,e_{n}\right\} $ satisfies the anti-commutation
relationship $e_{j}e_{k}+e_{k}e_{j}=-2\delta _{kj},$ where $1\leq j,$ $k\leq
n$ and $\delta _{kj}$ is the Kronecker delta symbol. The Clifford algebra
can be considered as being the exterior algebra $\Lambda 
\mathbb{R}
^{n}$ where the inner-product is added onto the Clifford product,that is, $%
xy=x\wedge y-\left\langle x,y\right\rangle ,$ for arbitrary elements $x,y\in 
\mathbb{R}
^{n}.$

\begin{definition}
The Dirac operator in $%
\mathbb{R}
^{n}$ acts on smooth Cl$_{n}$- valued functions and is defined by 
\begin{equation*}
D_{x}f\left( x\right) =\sum_{i=1}^{n}e_{i}\partial _{i}f\left( x\right) ,
\end{equation*}%
where $\partial _{i}$ is the i-th partial derivative.
\end{definition}

\begin{remark}
An important point is that the $D_{x}^{2}=-\nabla _{x},$ where $\nabla _{x}$
is the Laplacian,This follows from the fact that $x^{2}$ =$-\left\Vert
x\right\Vert ^{2}$ for $x\in 
\mathbb{R}
^{n}$. We say that a function f is left (resp. right) monogenic if $Df=0$
(resp. $fD=0$).
\end{remark}

We now give several examples of theorems which generalize from classical
complex analysis to Clifford analysis. To this end, consider a piecewise-$%
C^{1}$ surface in $%
\mathbb{R}
^{n}$, that is, a codimension-1 manifold embedded in $%
\mathbb{R}
^{n}$. Let $U$ $\subset 
\mathbb{R}
^{n}$ be a domain and suppose $V$ is relatively compact in $U$, denoted $%
V\subset \subset U$, such that $\partial V$ is a piecewise-$C^{1}$ surface.
Then we wish to evaluate an integral of the form:%
\begin{equation*}
\int g\left( x\right) \eta \left( x\right) f\left( x\right) d\sigma \left(
x\right) ,
\end{equation*}%
where f and g are Clifford algebra-valued functions, $\eta \left( x\right) $
is normal to the $\partial V$ at x$\in 
\mathbb{R}
^{n}.$

\begin{theorem}
(Cauchy's Theorem) Suppose $U\subset 
\mathbb{R}
^{n}$ is a domain and $V$ is relatively compact with $\partial V$ a
piecewise-$C^{1}$ hypersurface. Suppose $f,g:U\rightarrow Cl_{n}$ are right
and left monogenic functions, respectively. Then 
\begin{equation*}
\int g\left( x\right) \eta \left( x\right) f\left( x\right) d\sigma \left(
x\right) =0.
\end{equation*}
\end{theorem}

The function $G:%
\mathbb{R}
^{n}/\left\{ 0\right\} \rightarrow 
\mathbb{R}
^{n}/\left\{ 0\right\} $ defined by $G(x)=-\frac{x}{\left\Vert x\right\Vert
^{n}}$ is the Cauchy kernel in Clifford analysis. The function is infinitely
often continuously differentiable and it is homogeneous of degree 1-n with
respect to the origin. For more details, see for example, [5,12].

\begin{theorem}
(Cauchy's Integral Formula) Let $U\subset 
\mathbb{R}
^{n}$ be a domain and let $V$ be open and relatively compact in $U$ such
that $\partial V$ is a $C^{1}$ hypersurface. Suppose f is a left monogenic
function on $U$. Then for $y\in V,$%
\begin{equation*}
f\left( y\right) =\frac{1}{\omega _{n}}\int_{\partial V}G\left( x-y\right)
\eta \left( x\right) f\left( x\right) d\sigma \left( x\right) ,
\end{equation*}%
where $\omega _{n}$ is the surface area of the unit sphere $S^{n-1}$ in $%
\mathbb{R}
^{n}.$
\end{theorem}

Next we need to derive the Dirac operator on $%
\mathbb{R}
^{n+1}$ in polar form. For this we follow the approach in [5,12]. Consider
polar coordinates, $\left( r,\omega \right) \in 
\mathbb{R}
_{+}\times S^{n}$ in $%
\mathbb{R}
^{n+1},$ where $r=\left\vert x\right\vert =\left( x_{1}^{2}+\cdot \cdot
\cdot +x_{n+1}^{2}\right) ^{1/2},$ $\omega =\frac{x}{\left\vert x\right\vert 
}$ (a unit vector on S$^{n}).$ Then the Dirac operator admits the polar
decomposition%
\begin{equation*}
D_{x}=\frac{\omega }{r}\left( \Gamma _{\omega }+r\partial _{r}\right) .
\end{equation*}%
We write this as 
\begin{equation*}
D_{x}=\omega \left( \partial _{r}+\frac{1}{r}\Gamma _{\omega }\right) .
\end{equation*}%
Next we define the spinor connection on the sphere.

\begin{definition}
The spinor connection on the sphere in the x-direction is given $\nabla
_{x}=\partial _{x}+\frac{1}{2}x\omega $ where $\omega \in S^{n}$.
\end{definition}

In [4] it is shown via the spinor connection it is possible to derive the
conformal Dirac operator on the sphere. We state it here as a definition.

\begin{definition}
The conformal Dirac operator acting on smooth Clifford algebra-valued
functions on the sphere is given by 
\begin{equation*}
D_{\mathbf{s}}=\omega \left( \Gamma _{\omega }-\frac{n}{2}\right) ,
\end{equation*}%
where $\Gamma _{\omega }$ is the Dirac-Beltrami operator or the gamma
operator. See [4,11]
\end{definition}

The Dirac-Beltrami operator has the following properties.

\begin{lemma}
Let $\omega \in S^{n}.$ The we have the following intertwining operator
identity for the Dirac-Beltrami \ operator: $\Gamma _{\omega }\omega +\omega
\Gamma _{\omega }=\omega n.$
\end{lemma}

\begin{proof}
Using that $D_{x}^{2}=-\nabla _{n},$ where $\nabla _{n}$ is the Laplacian in 
$%
\mathbb{R}
^{n+1}$ given by 
\begin{equation*}
-\nabla _{n}=-\partial _{r}^{2}-\frac{n}{r}\partial _{r}-\frac{1}{r^{2}}%
\nabla _{\omega },
\end{equation*}%
$\nabla _{\omega }$ being up to sign the Laplace-Beltrami operator on $%
S^{n}. $ Using this and the commutation relations $\left[ \partial
_{r},\omega \right] =0$ and $\left[ \Gamma _{\omega },\partial _{r}\right]
=0,$ it follows that 
\begin{eqnarray*}
D_{x}^{2} &=&\omega \left( \partial _{r}+\frac{1}{r}\Gamma _{\omega }\right)
\omega \left( \partial _{r}+\frac{1}{r}\Gamma _{\omega }\right) \\
&=&-\partial _{r}^{2}+\left( \frac{\omega \Gamma _{\omega }\omega -\Gamma
_{\omega }}{r}\right) \partial _{r}+\left( \frac{\omega \Gamma _{\omega
}\omega \Gamma _{\omega }+\Gamma _{\omega }}{r}\right) .
\end{eqnarray*}%
Therefore, 
\begin{equation*}
\omega \Gamma _{\omega }\omega -\Gamma _{\omega }=-n\left( Id\right)
\end{equation*}%
so that 
\begin{equation*}
\Gamma _{\omega }\omega +\omega \Gamma _{\omega }=\omega n\left( Id\right) ,
\end{equation*}%
or simply 
\begin{equation*}
\Gamma _{\omega }\omega +\omega \Gamma _{\omega }=\omega n.
\end{equation*}
\end{proof}

\begin{remark}
In [3] and [4] it is shown that $D_{s}$ can act on smooth Clifford
algebra-valued functions or smooth spinor sections. We are considering that $%
D_{s}$ acts on smooth Clifford algebra-valued functions instead of smooth
spinor sections. For more details, see, [3].
\end{remark}

\section{ Spherical Harmonics and Orthogonal Projection on the Sphere}

\begin{definition}
Let $\mathbb{H}_{m}$ denote the restriction to $S^{n}$ of the space of $%
Cl_{n+1}-$valued harmonic polynomials homogeneous of degree $m\in 
\mathbb{N}
\cup \{0\}.$ This is the space of spherical harmonic polynomials homogenous
of degree m. Further, let $P_{m}$ denote the restriction to $S^{n}$ of left
monogenic polynomials homogeneous of degree $m\in 
\mathbb{N}
\cup \{0\},$ and let $Q_{m}$ denote the restriction to $S^{n}$ of the space
of left monogenic functions of degree $-n-m$ where m=0,1,2,....
\end{definition}

The set of left monogenic and right monogenic polynomials provided with the
obvious laws of pointwise addition and (right) multiplication by Clifford
numbers are right Clifford modules.

\begin{definition}
Let $\mathbb{P}_{a}\left( E\right) $ denote the space of $Cl_{n+1}-$valued
harmonic polynomials of degree $\leq a,$where $a\in 
\mathbb{N}
\cup \left\{ 0\right\} $ on $S^{n}.$ That is, the space of the restrictions
to $S^{n}$ of all homogeneous harmonic polynomials of degree $\leq a,$ where 
$a\in 
\mathbb{N}
\cup \left\{ 0\right\} $ on $%
\mathbb{R}
^{n+1}.$
\end{definition}

\begin{lemma}
The space $\mathbb{H}_{m}$ of $Cl_{n+1}-$valued spherical harmonic
polynomials homogeneous of degree $m\in 
\mathbb{N}
\cup \left\{ 0\right\} $ on $%
\mathbb{R}
^{n+1}$ has dimension%
\begin{equation*}
N\left( n,m\right) =\dim H_{m}=\dim \left( Cl_{n+1}\right) \left( \dim
P_{m}+\dim Q_{m}\right) ,
\end{equation*}%
where 
\begin{equation*}
\dim P_{m}=2^{\left[ \frac{n+1}{2}\right] \frac{\left( m+n-1\right) !}{%
m!(n-1)!}}
\end{equation*}%
and 
\begin{equation*}
\dim Q_{m}=2^{\left[ \frac{n+1}{2}\right] \frac{\left( m+n-1\right) !}{%
m!(n-1)!}}
\end{equation*}
\end{lemma}

\begin{proof}
See $\left[ 5\right] .$
\end{proof}

To make further progress note 
\begin{equation*}
N\left( n,0\right) =1\text{ and }N(n,m)=2^{\left[ \frac{n+1}{2}\right] +1}%
\frac{(2m+n-1)\left( m+n-2\right) !}{\left( n-1\right) !m!}
\end{equation*}

and denote by 
\begin{equation}
\left\{ Y_{mk}^{\left( n\right) }\mid k=1,...,N\left( n,m\right) \right\}
\end{equation}%
a fixed $L^{2}\left( S^{n}\right) $ orthonormal system of $Cl_{n+1}-$ valued
spherical harmonic polynomial homogeneous of degree m for $\mathbb{H}%
_{m}\left( S^{n}\right) .$ In $\left[ 9\right] $ it is shown that Eqn. (1)
is an $L^{2}(S^{n})$ orthonormal basis for $\mathbb{H}_{m}\left(
S^{n}\right) .$ Moreover, it is shown that 
\begin{equation*}
L^{2}(S^{n})=\oplus _{m=0}^{\infty }\mathbb{H}_{m}\left( S^{n}\right)
=\dsum\limits_{m=0}^{\infty }\left( P_{m}\oplus Q_{m}\right) \left(
S^{n}\right) =\cup _{a=0}^{\infty }\mathbb{H}_{a}\left( S^{n}\right) .
\end{equation*}%
Consequently, we have 
\begin{equation*}
\mathbb{P}_{a}\left( S^{n}\right) =\oplus _{m=0}^{a}\mathbb{H}_{m}\left(
S^{n}\right) =\dsum\limits_{m=0}^{a}\left( P_{m}\oplus Q_{m}\right) \left(
S^{n}\right)
\end{equation*}%
with 
\begin{eqnarray*}
d_{a,n} &:&=\dim \left( \mathbb{P}_{a}\left( E\right) \right) =2^{\left[ 
\frac{n+1}{2}\right] +1}\dsum\limits_{m=0}^{a}N\left( n,m\right) =2^{\left[ 
\frac{n+1}{2}\right] +1}N\left( n+1,a\right) \\
&=&2^{\left[ \frac{n+1}{2}\right] +1}\frac{\left( 2a+n\right) \left(
a+n-1\right) !}{n!a!}
\end{eqnarray*}%
In $\left[ 1\right] $ it is shown that any two spherical harmonics of
different degrees are orthogonal, and hence the union of the sets (2) over
all $m\in 
\mathbb{N}
_{0}$ is a complete orthonormal system in the Hilbert module $L^{2}\left(
S^{n}\right) .$ Therefore, a function $f\in L^{2}\left( S^{n}\right) $ can
be represented in the $L^{2}\left( S^{n}\right) $ sense by a generalized
Fourier series expansion with respect to this complete orthonormal system of 
$Cl_{n+1}-$ valued spherical harmonic polynomial homogeneous of degree $m:$%
\begin{equation*}
f=\dsum\limits_{m=0}^{\infty }\dsum\limits_{k=1}^{N(n,m)}\hat{f}%
_{mk}^{\left( n\right) }Y_{mk}^{\left( n\right) }
\end{equation*}%
with the generalized Fourier coefficients:%
\begin{equation*}
\hat{f}_{mk}^{\left( n\right) }=\left( f,Y_{mk}^{\left( n\right) }\right)
_{L^{2}\left( S^{n}\right) }=\int_{S^{n}}\overline{f\left( \omega \right) }%
Y_{mk}^{\left( n\right) }\left( \omega \right) dS_{n}\left( \omega \right) ,
\end{equation*}%
where $dS_{n}\left( \omega \right) $ is Lebesgue surface measure on $S^{n}.$

\section{ Polynomial Basis and Orthogonal Projection on the Unit Sphere}

\begin{definition}
Let $\mathbb{H}_{m}$ denote the restriction to $S^{n}$ of the space of $%
Cl_{n+1}-$valued harmonic polynomials homogeneous of degree $m\in 
\mathbb{N}
\cup \{0\}.$ This is the space of spherical harmonic polynomials homogenous
of degree m. Further, let $P_{m}$ denote the restriction to $S^{n}$ of left
monogenic polynomials homogeneous of degree $m\in 
\mathbb{N}
\cup \{0\},$ and let $Q_{m}$ denote the restriction to $S^{n}$ of the space
of left monogenic functions of degree $-n-m$ where \textit{m=0,1,2,...}.
\end{definition}

The set of left monogenic and right monogenic polynomials provided with the
obvious laws of pointwise addition and (right) multiplication by Clifford
numbers are right Clifford modules.

\begin{definition}
Let $\mathbb{P}_{a}\left( E\right) $ denote the space of $Cl_{n+1}-$valued
harmonic polynomials of degree $\leq a,$where $a\in 
\mathbb{N}
\cup \left\{ 0\right\} $ on $S^{n}.$ That is, the space of the restrictions
to $S^{n}$ of all homogeneous harmonic polynomials of degree $\leq a,$ where 
$a\in 
\mathbb{N}
\cup \left\{ 0\right\} $ on $%
\mathbb{R}
^{n+1}.$
\end{definition}

\begin{lemma}
The space $\mathbb{H}_{m}$ of $Cl_{n+1}-$valued spherical harmonic
polynomials homogeneous of degree $m\in 
\mathbb{N}
\cup \left\{ 0\right\} $ on $%
\mathbb{R}
^{n+1}$ has dimension%
\begin{equation*}
N\left( n,m\right) =\dim H_{m}=\dim \left( Cl_{n+1}\right) \left( \dim
P_{m}+\dim Q_{m}\right) ,
\end{equation*}%
where 
\begin{equation*}
\dim P_{m}=2^{\left[ \frac{n+1}{2}\right] \frac{\left( m+n-1\right) !}{%
m!(n-1)!}}
\end{equation*}%
and 
\begin{equation*}
\dim Q_{m}=2^{\left[ \frac{n+1}{2}\right] \frac{\left( m+n-1\right) !}{%
m!(n-1)!}}
\end{equation*}
\end{lemma}

\begin{proof}
See $\left[ 5\right] .$
\end{proof}

To make further progress note 
\begin{equation*}
N\left( n,0\right) =1\text{ and }N(n,m)=2^{\left[ \frac{n+1}{2}\right] +1}%
\frac{(2m+n-1)\left( m+n-2\right) !}{\left( n-1\right) !m!}
\end{equation*}

and denote by 
\begin{equation}
\left\{ Y_{mk}^{\left( n\right) }\mid k=1,...,N\left( n,m\right) \right\}
\end{equation}%
a fixed $L^{2}\left( S^{n}\right) $ orthonormal system of $Cl_{n+1}-$ valued
spherical harmonic polynomial homogeneous of degree m for $\mathbb{H}%
_{m}\left( S^{n}\right) .$ In $\left[ 4\right] $ it is shown that Eqn. (1)
is an $L^{2}(S^{n})$ orthonormal basis for $\mathbb{H}_{m}\left(
S^{n}\right) .$ Moreover, it is shown that 
\begin{equation*}
L^{2}(S^{n})=\oplus _{m=0}^{\infty }\mathbb{H}_{m}\left( S^{n}\right)
=\dsum\limits_{m=0}^{\infty }\left( P_{m}\oplus Q_{m}\right) \left(
S^{n}\right) =\cup _{a=0}^{\infty }\mathbb{H}_{a}\left( S^{n}\right) .
\end{equation*}%
Consequently, we have 
\begin{equation*}
\mathbb{P}_{a}\left( S^{n}\right) =\oplus _{m=0}^{a}\mathbb{H}_{m}\left(
S^{n}\right) =\dsum\limits_{m=0}^{a}\left( P_{m}\oplus Q_{m}\right) \left(
S^{n}\right)
\end{equation*}%
with 
\begin{eqnarray*}
d_{a,n} &:&=\dim \left( \mathbb{P}_{a}\left( E\right) \right) =2^{\left[ 
\frac{n+1}{2}\right] +1}\dsum\limits_{m=0}^{a}N\left( n,m\right) =2^{\left[ 
\frac{n+1}{2}\right] +1}N\left( n+1,a\right) \\
&=&2^{\left[ \frac{n+1}{2}\right] +1}\frac{\left( 2a+n\right) \left(
a+n-1\right) !}{n!a!}
\end{eqnarray*}%
In $\left[ 9\right] $ it is shown that any two spherical harmonics of
different degrees are orthogonal, and hence the union of the sets (1) over
all $m\in 
\mathbb{N}
_{0}$ is a complete orthonormal system in the Hilbert module $L^{2}\left(
S^{n}\right) .$ Therefore, a function $f\in L^{2}\left( S^{n}\right) $ can
be represented in the $L^{2}\left( S^{n}\right) $ sense by a generalized
Fourier series expansion with respect to this complete orthonormal system of 
$Cl_{n+1}-$ valued spherical harmonic polynomial homogeneous of degree $m:$%
\begin{equation*}
f=\dsum\limits_{m=0}^{\infty }\dsum\limits_{k=1}^{N(n,m)}\hat{f}%
_{mk}^{\left( n\right) }Y_{mk}^{\left( n\right) }
\end{equation*}%
with the generalized Fourier coefficients:%
\begin{equation*}
\hat{f}_{mk}^{\left( n\right) }=\left( f,Y_{mk}^{\left( n\right) }\right)
_{L^{2}\left( E\right) }=\int_{S^{n}}\overline{f\left( \omega \right) }%
Y_{mk}^{\left( n\right) }\left( \omega \right) dS_{n}\left( \omega \right) ,
\end{equation*}%
where $dS_{n}\left( \omega \right) $ is Lebesgue surface measure on $S^{n}.$
The surface area of $S^{n}$ in $%
\mathbb{R}
^{n+1}$ is denoted by $S_{n}\left( \omega \right) ,$ where 
\begin{equation*}
S_{n}=\left\vert S_{n}\right\vert =\frac{2\pi ^{\frac{n+1}{2}}}{\Gamma
\left( \frac{n+1}{2}\right) }.
\end{equation*}

\begin{definition}
The orthogonal projection operator $T_{a}:L^{2}\left( S^{n}\right)
\rightarrow \mathbb{P}_{a}\left( S^{n}\right) $ onto $\mathbb{P}_{a}\left(
S^{n}\right) $ may be represented by 
\begin{equation*}
T_{a}f=\sum_{m=0}^{a}\dsum\limits_{k=1}^{N(n,m)}\hat{f}_{mk}^{\left(
n\right) }Y_{mk}^{\left( n\right)
}=\sum_{m=0}^{a}\dsum\limits_{k=1}^{N(n,m)}\left( f,Y_{mk}^{\left( n\right)
}\right) _{L_{2}\left( S^{n}\right) }Y_{mk}^{\left( n\right) }.
\end{equation*}
\end{definition}

\bigskip Another form of $T_{a}$ can be given via the reproducing kernel of $%
\mathbb{P}_{a}\left( S^{n}\right) .$

\begin{definition}
The reproducing kernel of $\mathbb{P}_{a}\left( S^{n}\right) $ is the
uniquely determined kernel $G_{a}:S^{n}\times S^{n}\rightarrow 
\mathbb{R}
$ with the following properties : $\left( i\right) $ $G_{a}\left( \omega
,\cdot \right) \in \mathbb{P}_{a}\left( S^{n}\right) $ for every fixed $%
\omega \in S^{n},$ $\left( ii\right) $ $G_{a}\left( w,\upsilon \right) $ for
all $\omega ,\upsilon \in S^{n},$ and $\left( iii\right) $ the following
reproducing property%
\begin{equation*}
\left( f,G_{a}\left( \omega ,\cdot \right) \right) _{L^{2}\left(
S^{n}\right) }=f\left( \omega \right) \text{ }
\end{equation*}%
for all $\omega \in S^{n}$ and $f\in \mathbb{P}_{a}\left( S^{n}\right) .$
\end{definition}

Before stating the reproducing kernel $G_{a}$ we need the following
well-known result.

\begin{proposition}
The addition theorem for spherical harmonic polynomials homogeneous of
degree $m\in 
\mathbb{N}
_{0}$ is given by 
\begin{equation*}
\dsum\limits_{k=1}^{N(n,m)}Y_{mk}^{\left( n\right) }\left( \omega \right)
\otimes Y_{mk}^{\ast \left( n\right) }\left( \nu \right) =\frac{1}{\omega
_{n+1}}N(n,m)\frac{C_{m}^{\frac{n-1}{2}}\left( -\left\langle \omega
,\upsilon \right\rangle \right) }{C_{m}^{\frac{n-1}{2}}\left( 1\right) },
\end{equation*}%
where $C_{m}^{\frac{n-1}{2}}$ is the ultra-spherical or Gegenbauer
polynomial of degree $m$ with index $\lambda =\frac{n-1}{2},$ where 
\begin{equation*}
C_{m}^{\lambda }\left( t\right) =\frac{\left( 2\lambda \right) _{m}}{\left(
\lambda +1\right) _{m}}P_{m}^{\left( \lambda -\frac{1}{2},\lambda -\frac{1}{2%
}\right) }\left( t\right) ,\text{ }t\in \left[ -1,1\right] .
\end{equation*}%
with $\left( a\right) _{0}=1,$ $\left( a\right) _{l}=a\left( a+1\right)
\cdot \cdot \cdot \left( a+l-1\right) ,$ $a\in 
\mathbb{R}
,$ $l\in 
\mathbb{N}
,$ and $P_{m}^{\left( \lambda -\frac{1}{2},\lambda -\frac{1}{2}\right) }:%
\left[ -1,1\right] \rightarrow 
\mathbb{R}
$ is the Jacobi polynomial $P_{m}^{\left( \alpha ,\beta \right) }$ of degree
m with indices $\alpha =\beta =\lambda -\frac{1}{2}.$ Note that $\omega
_{n+1}$ is the surface area of the unit sphere $S^{n}$ in $%
\mathbb{R}
^{n+1}.$
\end{proposition}

See $\left[ 9,\text{ p. 10}\right] .$

Therefore, we can write the $L^{2}\left( S^{n}\right) -$orthogonal
projection operator $T_{a}$ onto $\mathbb{P}_{a}\left( S^{n}\right) $ with
the reproducing kernel $G_{a}$ of $\mathbb{P}_{a}\left( S^{n}\right) $ as%
\begin{equation*}
T_{a}f\left( \omega \right) =\int_{S^{n}}f\left( v\right) G_{a}\left( \omega
,\nu \right) dS_{n}\left( \nu \right) =\left( f,G_{a}\left( \omega ,\cdot
\right) \right) _{L^{2}\left( S^{n}\right) },
\end{equation*}%
which is a form of the Cauchy Integral Formula in Clifford Analysis.

\begin{definition}
For $d\in 
\mathbb{N}
,$ the spinorial Laplacian of order $d>0$ on S$^{n}$ is the operator defined
by 
\begin{equation*}
\left( D_{\mathbf{s}}^{2}\right) ^{d}=\left( \Gamma _{\omega }-\frac{n}{2}%
\right) ^{2d}
\end{equation*}
\end{definition}

We need the following lemma from Balinski and Ryan [1]. We reproduce the
proof here for completeness.

\begin{lemma}
The spectral resolution of the conformal Dirac operator $D_{s}$ on the unit
sphere $S^{n}$ is given by 
\begin{equation*}
\sigma \left( D_{\mathbf{s}}\right) =\left\{ \left( m+\frac{n}{2}\right)
;m=0,1,2,...\cup -\left( m+\frac{n}{2}\right) ;m=0,1,2,...\right\}
\end{equation*}%
and where the eigenvectors are $p_{m}\in P_{m}$ and $q_{m}\in Q_{m}.$
\end{lemma}

\begin{proof}
Let $\phi :S^{n}\rightarrow Cl_{n+1}$ be a $C^{1}$ function. Then $\phi \in
L^{2}\left( S^{n}\right) $ and 
\begin{equation*}
\phi \left( \omega \right) =\sum_{m=0}^{\infty }\sum_{j=0}^{\dim
P_{m}}p_{mj}\left( \omega \right) +\sum_{m=0}^{-\infty }\sum_{j=0}^{\dim
Q_{m}}q_{mj}\left( \omega \right) ,
\end{equation*}%
where $p_{mj}\in P_{m}$ and $q_{mj}\in Q_{m}$ are eigenvectors of the
Dirac-Beltrami operator $\Gamma _{\omega }.$ Further they may be chosen so
that within $P_{m}$ they are mutually orthogonal. This is similar for the
eigenvectors in $Q_{m}.$ As $\phi \in C^{1}$ then $D_{s}\phi \in C^{0}\left(
S^{n}\right) $ and so $D_{s}\phi \in L^{2}\left( S^{n}\right) .$
Consequently, 
\begin{equation*}
D_{\mathbf{s}}\phi =\omega \left( \sum_{m=0}^{\infty }\left( m+\frac{n}{2}%
\right) \sum_{j=0}^{\dim P_{m}}p_{mj}\left( \omega \right)
+\sum_{m=0}^{\infty }\left( -m-\frac{n}{2}\right) \sum_{j=0}^{\dim
Q_{m}}q_{mj}\left( \omega \right) \right) .
\end{equation*}%
But $\omega p_{m}\left( \omega \right) \in Q_{m}$ and $\omega q_{m}\left(
\omega \right) \in P_{m}.$ Consequently,%
\begin{equation*}
D_{\mathbf{s}}\phi =\sum_{m=0}^{\infty }\left( m+\frac{n}{2}\right)
\sum_{j=0}^{\dim Q_{m}}q_{mj}\left( \omega \right) +\sum_{m=0}^{\infty
}\left( -m-\frac{n}{2}\right) \sum_{j=0}^{\dim P_{m}}p_{mj}\left( \omega
\right) .
\end{equation*}%
Therefore, the spectrum, $\sigma \left( D_{\mathbf{s}}\right) ,$ of the
conformal Dirac operator D$_{s}$ on the unit sphere $S^{n}$ is given by 
\begin{equation*}
\sigma \left( D_{\mathbf{s}}\right) =\left\{ \left( m+\frac{n}{2}\right)
;m=0,1,2,...\cup -\left( m+\frac{n}{2}\right) ;m=0,1,2,...\right\} .
\end{equation*}
\end{proof}

\section{Sobolev Space Estimates in Clifford Analysis on the Unit Sphere}

The embedding theorems include various inclusions between the Sobolev spaces 
$L_{s}^{2}.$

\begin{definition}
For any non-negative integer $k$, let $C^{k}\left( S^{n}\right) $ denote the
space of $k-$times continuously differentiable functions on $S^{n}$ equipped
with the uniform $C^{k}-norm,$ defined for $\varphi \in $ $C^{k}\left(
S^{n}\right) $ by 
\begin{equation*}
\left\Vert \varphi \left( \omega \right) \right\Vert _{C^{k}\left(
S^{n}\right) }^{2}=\sum_{j=0}^{k}\sup_{\omega \in TS^{n}\backslash \left\{
0\right\} }\left\vert \nabla ^{j}\varphi \left( \omega \right) \right\vert
^{2},
\end{equation*}%
Note that $TS^{n}$ is the tangent space at $\omega \in S^{n}$ and $\nabla $
is the spinor connection. The subspace of compactly supported functions is
denoted by $C_{0}^{k}\left( S^{n}\right) .$
\end{definition}

\begin{definition}
For a non-negative real number $s$, we define the Sobolev spaces $%
L_{s}^{2}\left( S^{n}\right) $ as the closure of $\oplus _{m=0}^{\infty }%
\mathbb{H}\left( S^{n}\right) $ with respect to the $s^{th}$ Sobolev norm%
\begin{equation*}
\left\Vert \varphi \right\Vert _{L_{s}^{2}\left( S^{n}\right)
}^{2}=\sum_{m=0}^{\infty }\left( m+\frac{n-1}{2}\right)
^{2s}\sum_{k=1}^{N\left( n,m\right) }\left\vert \hat{\varphi}_{mk}^{\left(
n\right) }\right\vert ^{2}
\end{equation*}%
The space L$_{s}^{2}$ is a Hilbert space with the inner product%
\begin{equation*}
\left( \varphi ,\psi \right) _{L_{s}^{2}\left( S^{n}\right)
}=\sum_{m=0}^{\infty }\left( m+\frac{n-1}{2}\right) ^{2s}\sum_{k=1}^{N\left(
n,m\right) }\hat{\varphi}_{m,k}^{\left( n\right) }\hat{\psi}_{m,k}^{\left(
n\right) }\text{ for }\varphi ,\psi \in L_{s}^{2}\left( S^{n}\right) ,
\end{equation*}%
which induces the norm $\left\Vert \cdot \right\Vert _{L_{s}^{2}\left(
S^{n}\right) }.$ Thus 
\begin{equation*}
L_{s}^{2}\left( S^{n}\right) =\left\{ \varphi \in S^{n}:\left\Vert \varphi
\right\Vert _{L_{s}^{2}\left( S^{n}\right) }^{2}<\infty \right\} ,
\end{equation*}%
where $\varphi =\sum_{m=0}^{\infty }\sum_{k=1}^{N\left( n,m\right) }\hat{%
\varphi}_{mk}^{\left( n\right) }Y_{mk}^{\left( n\right) }$ with inner
product $\left\langle \varphi ,\psi \right\rangle _{L_{s}^{2}\left(
S^{n}\right) }.$
\end{definition}

\begin{remark}
For $s>\frac{n}{2}$ there exists a constant $c_{s}$ such that 
\begin{equation*}
\left\Vert f\right\Vert _{C\left( S^{n}\right) }\leq c_{s}\left\Vert
f\right\Vert _{L_{s}^{2}\left( S^{n}\right) },\text{ }\forall f\in
L_{s}^{2}\left( S^{n}\right) ,
\end{equation*}%
that is, $L_{s}^{2}\left( S^{n}\right) $ is embedded in $C\left(
S^{n}\right) .$ Also, for $s>\frac{n}{2}$ the space $L_{s}^{2}\left(
S^{n}\right) $ is a reproducing kernel Hilbert space. Note that the spaces $%
L_{s}^{2}\left( S^{n}\right) $ are nested, that is, $L_{t}^{2}\left(
S^{n}\right) \subset L_{s}^{2}\left( S^{n}\right) $ whenever $t\geq s.$
\end{remark}

\begin{definition}
We denote by 
\begin{equation*}
\mathbb{P}_{a}^{\bot }\left( S^{n}\right) =\left\{ f\in L^{2}\left(
S^{n}\right) :\left( f,g\right) _{L_{2}\left( S^{n}\right) }=0\text{ }%
\forall g\in \mathbb{P}_{a}\left( S^{n}\right) \right\}
\end{equation*}%
the orthogonal complement of $\mathbb{P}_{a}\left( S^{n}\right) $ in $%
L^{2}\left( S^{n}\right) .$
\end{definition}

The orthogonal complement of $\mathbb{P}_{a}\left( S^{n}\right) $ in $%
L_{s}^{2}\left( S^{n}\right) ,$that is the space of all those functions in $%
L^{2}\left( S^{n}\right) $ which are $L^{2}\left( S^{n}\right) -$orthogonal
to $\mathbb{P}_{a}\left( S^{n}\right) .$ From the definition of the inner
product $\left( \cdot ,\cdot \right) _{L_{s}^{2}\left( S^{n}\right) },$ the
orthogonal complement of $\mathbb{P}_{a}\left( S^{n}\right) $ in $%
L_{t}^{2}\left( S^{n}\right) $ is simply $\mathbb{P}_{a}^{\bot }\left(
S^{n}\right) $ $\cap L_{t}^{2}\left( S^{n}\right) .$

The following estimates hold for either sections in $\mathbb{P}_{a}\left(
S^{n}\right) $ or in $\mathbb{P}_{a}^{\bot }\left( S^{n}\right) $ $\cap
L_{s}^{2}\left( S^{n}\right) .$

\begin{lemma}
The following estimates hold in the Sobolev spaces $L_{s}^{2}\left(
S^{n}\right) .$

\begin{enumerate}
\item Let $s\geq 0.$ Let $\varphi \in \mathbb{P}_{a}\left( S^{n}\right) .$
Then $L_{0}^{2}\left( S^{n}\right) \subset L_{s}^{2}\left( S^{n}\right) ,$
that is 
\begin{equation*}
\left\Vert \varphi \right\Vert _{L_{s}^{2}\left( S^{n}\right) }\leq \left( d+%
\frac{n-1}{2}\right) ^{s}\left\Vert \varphi \right\Vert _{\left\Vert \varphi
\right\Vert _{L_{0}^{2}\left( S^{n}\right) }}.
\end{equation*}

\item Let $s\geq t\geq 0.$ Then for any $\varphi \in \mathbb{P}_{a}^{\bot
}\left( S^{n}\right) $ $\cap L_{s}^{2}\left( S^{n}\right) ,$ 
\begin{equation*}
\left\Vert \varphi \right\Vert _{L_{t}^{2}\left( S^{n}\right) }\leq \left( d+%
\frac{n-1}{2}\right) ^{t-s}\left\Vert \varphi \right\Vert _{L_{s}^{2}\left(
S^{n}\right) }.
\end{equation*}

\item Let k be a non-negative integer and let \ $s-\frac{n}{2}\geq k$\ .
Then $L_{s}^{2}\left( S^{n}\right) $ embeds continuously in $C^{k}\left(
S^{n}\right) ,$ that is, 
\begin{equation*}
L_{s}^{2}\left( S^{n}\right) \subset C^{k}\left( S^{n}\right) .
\end{equation*}%
If, moreover, $s-\frac{n}{2}>k,$ then the embedding is also compact.
\end{enumerate}
\end{lemma}

\begin{proof}
For (1), let $\varphi \in \mathbb{P}_{a}\left( S^{n}\right) .$ Then we have 
\begin{eqnarray*}
\left\Vert \varphi \right\Vert _{L_{s}^{2}\left( S^{n}\right) }^{2}
&=&\sum_{m=0}^{a}\sum_{j=1}^{N\left( n,m\right) }\left( m+\frac{n-1}{2}%
\right) ^{2s}\left\vert \hat{\varphi}_{mj}^{\left( n\right) }\right\vert ^{2}
\\
&\leq &\left( a+\frac{n-1}{2}\right) ^{2s}\left\Vert \varphi \right\Vert
_{L_{0}^{2}\left( S^{n}\right) }
\end{eqnarray*}%
pointwise. For (2), for any $\varphi \in \mathbb{P}_{a}^{\bot }\left(
S^{n}\right) $ $\cap L_{s}^{2}\left( S^{n}\right) $ we have for $s\geq t\geq
0,$%
\begin{eqnarray*}
\left\Vert \varphi \right\Vert _{L_{t}^{2}\left( S^{n}\right) }^{2}
&=&\sum_{m=a+1}^{\infty }\sum_{j=1}^{N\left( n,m\right) }\left( m+\frac{n-1}{%
2}\right) ^{2t}\left\vert \hat{\varphi}_{mj}^{\left( n\right) }\right\vert
^{2} \\
&\leq &\left( a+1+\frac{n-1}{2}\right) ^{2\left( t-s\right)
}\sum_{m=a+1}^{\infty }\sum_{j=1}^{N\left( n,m\right) }\left( m+\frac{n-1}{2}%
\right) ^{2s}\left\vert \hat{\varphi}_{mj}^{\left( n\right) }\right\vert ^{2}
\\
&\leq &\left( a+1+\frac{n-1}{2}\right) ^{2\left( t-s\right) }\left\Vert
\varphi \right\Vert _{L_{s}^{2}\left( S^{n}\right) }^{2},
\end{eqnarray*}%
which implies that the embedding is bounded.

To show that the embedding is compact, let I denote the inclusion map from $%
L_{s}^{2}\left( S^{n}\right) $ to $L_{t}^{2}\left( S^{n}\right) .$ Then for $%
\varphi \in L_{t}^{2}\left( S^{n}\right) ,$ we have 
\begin{equation*}
\left\Vert \left( T_{a}-I\right) \varphi \right\Vert _{L_{t}^{2}\left(
S^{n}\right) }\leq \left( a+\frac{n-1}{2}\right) ^{t-s}\left\Vert \varphi
\right\Vert _{L_{s}^{2}\left( S^{n}\right) }
\end{equation*}%
since $\left( T_{a}-I\right) \varphi \in \mathbb{P}_{a}^{\bot }\left(
S^{n}\right) $ $\cap L_{s}^{2}\left( S^{n}\right) .$ Observe that as a$%
\rightarrow \infty ,$ $\left\Vert \left( T_{a}-I\right) \varphi \right\Vert
_{L_{t}^{2}\left( S^{n}\right) }\rightarrow 0$. Now since the limit of
finite rank operators is a compact operator we see that the embedding of $%
L_{s}^{2}\left( S^{n}\right) $ into $L_{t}^{2}\left( S^{n}\right) $ is
compact.

For (3), consider first the case $k=0$. We must estimate the sup-norm of a
smooth function on $S^{n}$ in terms of the $L_{s}^{2}\left( S^{n}\right) $
norm. Let $\varphi \in \mathbb{P}_{a}^{\bot }\left( S^{n}\right) $ $\cap
L_{s}^{2}\left( S^{n}\right) .$ Then we have%
\begin{equation*}
\varphi \left( \omega \right) =\sum_{m=a+1}^{\infty }\sum_{j=1}^{N\left(
n,m\right) }\hat{\varphi}_{mj}^{\left( n\right) }Y_{mj}^{\left( n\right)
}\left( \omega \right) ,\text{ }\omega \in S^{n}.
\end{equation*}%
Assume this Fourier series is uniformly convergent when $s>\frac{n}{2},$ so
that it is also true in the pointwise sense. Hence, it follows from the
Cauchy-Schwarz inequality that%
\begin{eqnarray*}
\left\vert \varphi \left( \omega \right) \right\vert &=&\left\vert
\sum_{m=a+1}^{\infty }\sum_{j=1}^{N\left( n,m\right) }\hat{\varphi}%
_{mj}^{\left( n\right) }Y_{mj}^{\left( n\right) }\left( \omega \right)
\right\vert \\
&\leq &\sum_{m=a+1}^{\infty }\sum_{j=1}^{N\left( n,m\right) }\left[ \left( m+%
\frac{n-1}{2}\right) ^{s}\left\vert \hat{\varphi}_{mj}^{\left( n\right)
}\right\vert \right] \left[ \left( m+\frac{n-1}{2}\right) ^{-s}\left\vert
Y_{mj}^{\left( n\right) }\left( \omega \right) \right\vert \right] \\
&\leq &\sum_{m=a+1}^{\infty }\sum_{j=1}^{N\left( n,m\right) }\left[ \left( m+%
\frac{n-1}{2}\right) ^{2s}\left\vert \hat{\varphi}_{mj}^{\left( n\right)
}\right\vert ^{2}\right] ^{1/2}\sum_{m=a+1}^{\infty }\sum_{j=1}^{N\left(
n,m\right) }\left[ \left( m+\frac{n-1}{2}\right) ^{-2s}\left\vert
Y_{mj}^{\left( n\right) }\left( \omega \right) \right\vert ^{2}\right] ^{1/2}
\\
&\leq &\left\Vert \varphi \right\Vert _{L_{s}^{2}\left( S^{n}\right) }\left(
\sum_{m=a+1}^{\infty }\sum_{j=1}^{N\left( n,m\right) }\left[ \left( m+\frac{%
n-1}{2}\right) ^{-2s}\left\vert Y_{mj}^{\left( n\right) }\left( \omega
\right) \right\vert ^{2}\right] \right) \\
&=&\left\Vert \varphi \right\Vert _{L_{s}^{2}\left( S^{n}\right) }\left( 
\frac{1}{\omega _{n}}\sum_{m=a+1}^{\infty }\frac{N\left( n,m\right) }{\left(
m+\frac{n-1}{2}\right) ^{2s}}\right) ^{1/2}.
\end{eqnarray*}%
Observe that there exists positive constants $c_{1},c_{2}$ independent of $m$
such that 
\begin{equation*}
c_{1}\left( m+\frac{n-1}{2}\right) ^{n-1}\leq N(n,m)\leq c_{2}\left( m+\frac{%
n-1}{2}\right) ^{n-1}.
\end{equation*}%
By the integral test, the second term in the sum is finite if an only if $%
s>n/2.$

Furthermore, the sum satisfies%
\begin{equation*}
c_{3}\left( m+\frac{n-1}{2}\right) ^{n-2s}\leq \frac{1}{\omega _{n}}%
\sum_{m=a+1}^{\infty }\frac{N\left( n,m\right) }{\left( m+\frac{n-1}{2}%
\right) ^{2s}}\leq c_{4}\left( m+\frac{n-1}{2}\right) ^{n-2s}
\end{equation*}%
since 
\begin{equation*}
c_{3}\left( m+\frac{n-1}{2}\right) ^{n-2s}\leq \lim_{M\rightarrow \infty
}\int_{a}^{M}\left( x+\frac{n-1}{2}\right) ^{n-2s-1}dx\leq c_{4}\left( m+%
\frac{n-1}{2}\right) ^{n-2s}
\end{equation*}%
for appropriate constants $c_{3},c_{4}$ independent of $a$. Therefore, for $%
s>n/2$ with the positive constant $c_{5}=\sqrt{c_{4}}$%
\begin{equation*}
\sup_{\omega \in S^{n}}\left\vert \varphi \left( \omega \right) \right\vert
\leq c_{5}\left( a+\frac{n-1}{2}\right) ^{n-2s}\left\Vert \varphi
\right\Vert _{L_{s}^{2}\left( S^{n}\right) }
\end{equation*}%
which is the case for $k=0.$

For the general case, fix $k$ and choose $\varphi \in \mathbb{P}_{a}^{\bot
}\left( S^{n}\right) $ $\cap L_{s}^{2}\left( S^{n}\right) $ with $s-n/2\geq
k.$ For any $\alpha \leq k$ where $\alpha \in 
\mathbb{N}
,$ we have 
\begin{equation*}
\left\Vert D_{\mathbf{s}}^{\alpha }\varphi \left( \omega \right) \right\Vert
_{L_{s-\alpha }^{2}\left( S^{n}\right) }=\left\Vert \sum_{m=a+1}^{\infty
}\sum_{j=1}^{N\left( n,m\right) }\hat{\varphi}_{mj}^{\left( n\right) }\left(
D_{\mathbf{s}}^{\alpha }Y_{mj}^{\left( n\right) }\left( \omega \right)
\right) \right\Vert _{L_{s-\alpha }^{2}\left( S^{n}\right) }
\end{equation*}%
\begin{eqnarray*}
&\leq &\sum_{m=a+1}^{\infty }\sum_{j=1}^{N\left( n,m\right) }\left[ \left( m+%
\frac{n-1}{2}\right) ^{s}\left\vert \hat{\varphi}_{mj}^{\left( n\right)
}\right\vert \right] \left[ \left( m+\frac{n-1}{2}\right) ^{-s}\left( m+%
\frac{n}{2}\right) ^{\alpha }\left\vert Y_{mj}^{\left( n\right) }\left(
\omega \right) \right\vert \right] \\
&\leq &\sum_{m=a+1}^{\infty }\sum_{j=1}^{N\left( n,m\right) }\left[ \left( m+%
\frac{n-1}{2}\right) ^{2s}\left\vert \hat{\varphi}_{mj}^{\left( n\right)
}\right\vert ^{2}\right] ^{1/2}\cdot \left[ \left( m+\frac{n-1}{2}\right)
^{-2s}\left( m+\frac{n}{2}\right) ^{2\alpha }\left\vert Y_{mj}^{\left(
n\right) }\left( \omega \right) \right\vert ^{2}\right] ^{1/2} \\
&\leq &\left\Vert \varphi \right\Vert _{L_{s}^{2}\left( S^{n}\right) }\left( 
\frac{1}{\omega _{n}}\sum_{m=a+1}^{\infty }\frac{N\left( n,m\right) \left( m+%
\frac{n}{2}\right) ^{2\alpha }}{\left( m+\frac{n-1}{2}\right) ^{2s}}\right)
^{1/2}.
\end{eqnarray*}%
Observe that there exists constants $\hat{c}_{1}$ and $\hat{c}_{2}$ such
that 
\begin{equation*}
\hat{c}_{1}\left( m+\frac{n-1}{2}\right) ^{\alpha }\leq \left( m+\frac{n}{2}%
\right) ^{\alpha }\leq \hat{c}_{2}\left( m+\frac{n-1}{2}\right) ^{\alpha }
\end{equation*}%
for all $\alpha \in 
\mathbb{N}
.$ As before, there exists positive constants $\hat{c}_{3}$ and \ $\hat{c}%
_{4}$ such that $\hat{c}_{3}$%
\begin{equation*}
\hat{c}_{3}\left( m+\frac{n-1}{2}\right) ^{n-2s+2\alpha }\leq \frac{1}{%
\omega _{n}}\sum_{m=a+1}^{\infty }\frac{N\left( n,m\right) }{\left( m+\frac{%
n-1}{2}\right) ^{2s-2\alpha }}\leq \hat{c}_{4}\left( m+\frac{n-1}{2}\right)
^{n-2s-2\alpha }.
\end{equation*}%
Therefore, for $\alpha <s-n/2,$ we have 
\begin{equation*}
\left\Vert D_{\mathbf{s}}^{\alpha }\varphi \left( \omega \right) \right\Vert
_{L_{s-\alpha }^{2}\left( S^{n}\right) }\leq \hat{c}_{5}\left( a+\frac{n-1}{2%
}\right) ^{\frac{n}{2}-s+\alpha }\left\Vert \varphi \right\Vert
_{L_{s}^{2}\left( S^{n}\right) };
\end{equation*}%
with positive constants $\hat{c}_{5}=\sqrt{\hat{c}_{4}}.$ This implies that $%
D_{\mathbf{s}}^{\alpha }:L_{s}^{2}\left( S^{n}\right) \rightarrow
L_{s-\alpha }^{2}\left( S^{n}\right) $ is continuous. By the first part of
the proof, $D_{\mathbf{s}}^{\alpha }\varphi \in C^{0}\left( S^{n}\right) $
for all $\alpha \leq k,$ so that $f\in C^{k}\left( S^{n}\right) .$
\end{proof}

The following estimates should be able to be used in polynomial
interpolation on the unit sphere in Clifford analysis. For more details,
see, $\left[ 6\right] .$

\end{document}